\documentclass[11pt]{amsart}
\usepackage{amsmath}
\usepackage{amsthm}
\usepackage{mathrsfs}
\usepackage{amssymb,graphics,amscd,amsfonts}
\usepackage{color}
\usepackage{comment}
\usepackage{enumerate}
\usepackage{booktabs}
\usepackage{float}
\usepackage{mathtools}

\makeatletter

\@addtoreset{equation}{section}
\makeatother

\makeatletter
\@namedef{subjclassname@2020}{%
  \textup{2020} Mathematics Subject Classification}
\makeatother

\def\XXint#1#2#3{{\setbox0=\hbox{$#1{#2#3}{\int}$}
\vcenter{\hbox{$#2#3$}}\kern-.5\wd0}}

\theoremstyle{plain}
\newtheorem{theorem}{Theorem}[section]

\newtheorem{corollary}[theorem]{Corollary}
\newtheorem{lemma}[theorem]{Lemma}

\newtheorem*{problem}{Problem}

\theoremstyle{definition}
\newtheorem*{acknowledgements}{Acknowledgements}

\DeclareMathOperator{\Ric}{Ric}

\DeclareMathOperator{\grad}{grad}

\DeclareMathOperator{\Isom}{Isom}
\DeclareMathOperator{\Fix}{Fix}
\DeclareMathOperator{\Lie}{Lie}

\title[Extremal K\"ahler--Ricci solitons on Fano manifolds are K\"ahler--Einstein]{Extremal K\"ahler--Ricci solitons on \\Fano manifolds are K\"ahler--Einstein}

\author{Yasufumi Nitta}

\address{Department of Mathematics, Faculty of Science Division II, Tokyo University of Science, 1-3 Kagurazaka, Shinjuku-ku, Tokyo 162-8601, Japan}
\email{nitta@rs.tus.ac.jp}


\begin{document}
\pagestyle{plain}
\begin{abstract}
We prove that every extremal K\"ahler--Ricci soliton on a Fano manifold is K\"ahler--Einstein. 
This solves the problem of Calamai and Petrecca in full generality. 
\end{abstract}
\maketitle
\section{Introduction}
Finding canonical K\"ahler metrics in a given K\"ahler class is one of the central problems in K\"ahler geometry. 
In this paper, we consider relations among various canonical K\"ahler metrics on Fano manifolds. 
Let $X$ be an $n$-dimensional Fano manifold. 
Let $g$ be a K\"ahler metric on $X$ with its K\"ahler form $\omega_{g}$ representing $c_{1}(X)$. 
By the well-known $\partial\bar{\partial}$-lemma, there is a unique $F_{g} \in C^{\infty}(X, \mathbf{R})$ satisfying 
\begin{align}\label{Ricci_potential}
\Ric(\omega_{g})-\omega_{g} = \frac{\sqrt{-1}}{2\pi}\partial\bar{\partial}F_{g}\quad \text{and}\quad \int_{X}(1-e^{F_{g}})\omega_{g}^{n} = 0. 
\end{align}
The function $F_{g}$ is called the \emph{Ricci potential} of $g$. 
Among the various notions of canonical K\"ahler metrics, \emph{K\"ahler--Einstein metrics} form one of the most fundamental and well-studied classes. 
They are characterized by 
\begin{align*}
\Ric(\omega_{g}) = \omega_{g}, 
\end{align*}
or equivalently, by $F_{g} = 0$. 

In this paper, we consider two well-known generalizations of K\"ahler--Einstein metrics: extremal K\"ahler metrics and K\"ahler--Ricci solitons. 
The former were introduced by Calabi \cite{Ca82, Ca85}. 
A K\"ahler metric $g$ is called \emph{extremal} if $\grad_{\omega_g}^{\mathbf{C}} s(\omega_g)$ is a holomorphic vector field on $X$, where $s(\omega_g)$ denotes the scalar curvature of $g$. 
The latter, K\"ahler--Ricci solitons, were introduced by Koiso \cite{Koi90}. 
A K\"ahler metric $g$ is called a \emph{K\"ahler--Ricci soliton} if there exists a holomorphic vector field $V$ on $X$ such that
\[
\Ric(\omega_{g}) - \omega_{g} = L_{V}\omega_{g}. 
\]
Here $L_{V}$ means the Lie derivative. 
This condition is equivalent to saying that $\grad_{\omega_{g}}^{\mathbf{C}}F_{g}$ is a holomorphic vector field on $X$. 

It is clear that every K\"ahler--Einstein metric is both extremal and a K\"ahler--Ricci soliton. 
We are concerned with the converse implication. 
In \cite{CP16}, Calamai and Petrecca proposed the following problem. 

\begin{problem}[\cite{CP16}]
Prove that every extremal K\"ahler--Ricci soliton is Einstein or find a counterexample. 
\end{problem}

The authors gave in \cite{CP16} an affirmative answer to this problem under the assumption of positive holomorphic sectional curvature. 
They subsequently proved in \cite{CP17} that every extremal K\"ahler--Ricci soliton on a toric Fano manifold is necessarily K\"ahler--Einstein. 
Later, Lian extended this result to the case of Fano homogeneous toric bundles \cite{L24}. 
A similar result for radial K\"ahler metrics is known in \cite{LSZ21}. 

The main result of this paper gives a complete affirmative answer to the above problem. 

\begin{theorem}\label{main_thm}
Every extremal K\"ahler--Ricci soliton on a Fano manifold is K\"ahler--Einstein. 
\end{theorem}

We also note that \cite{N23} established closely related rigidity results concerning Mabuchi's generalized K\"ahler--Einstein metrics (also referred to as Mabuchi solitons in the literature), both in relation to extremal K\"ahler metrics and to K\"ahler--Ricci solitons. 
Analogous results in the more general setting of $\sigma$-extremal K\"ahler metrics and $\sigma$-solitons were subsequently obtained in \cite{NN25}. 

As a consequence of Theorem~\ref{main_thm}, we have the following corollary. 

\begin{corollary}
Let $X$ be a Fano manifold. 
If $X$ does not admit a K\"ahler--Einstein metric, then no extremal K\"ahler metric on $X$ can be holomorphically isometric to a K\"ahler--Ricci soliton. 
\end{corollary}

This paper is organized as follows. In Section 2, we briefly fix the notation and conventions used throughout the paper. 
In Section~3, we prove Theorem~\ref{main_thm}. 
The key of the proof is to combine the conservation law for K\"ahler--Ricci solitons with the convexity theorem for Hamiltonian torus actions, which forces the squared norm of the gradient of the Ricci potential to vanish identically. 
Finally, in Section 4, we explain the Sasakian analogue of our argument and obtain corresponding results for Sasaki--Ricci solitons and Sasaki-extremal metrics. 

\begin{acknowledgements}
The author was supported by JSPS KAKENHI Grant Number JP26K06790. 
\end{acknowledgements}

\section{Notation and conventions}
In this paper, we use the following notation and conventions. 
Let $(X, J, g)$ be an $n$-dimensional compact K\"ahler manifold with complex structure $J$ and K\"ahler metric $g$. 

\begin{itemize}
\item The K\"ahler form $\omega_{g}$ of $g$ is given by 
\begin{align*}
\omega_{g}(V, W) \coloneqq \frac{1}{2\pi}g(JV, W). 
\end{align*}
\item The pointwise scalar product on the space of complex differential forms induced from the K\"ahler metric is still denoted by $g$. 
\item We denote by $\bar{\partial}^{\ast}$ the formal adjoint of $\bar{\partial}$ with respect to the $L^{2}$-Hermitian inner product induced by $g$.
\item We denote by $\Delta_{\omega_g} \coloneqq \bar{\partial}\bar{\partial}^{\ast} + \bar{\partial}^{\ast}\bar{\partial}$ the $\bar{\partial}$-Laplacian. 
In local holomorphic coordinates, it  is given by $\Delta_{\omega_{g}}u = -\nabla^{i}\nabla_{i}u$ for each $u \in C^{\infty}(X, \mathbf{C})$. 
\item For each $u\in C^{\infty}(X,\mathbf{C})$, we denote by $\grad_{\omega_g}^{\mathbf{C}}u$ the $(1,0)$-part of $\grad_g u$. 
In local holomorphic coordinates, it is given by 
\[
\grad_{\omega_g}^{\mathbf{C}}u
=
\nabla^{i}u\frac{\partial}{\partial z^{i}}, 
\]
and satisfies
\[
i_{\grad_{\omega_g}^{\mathbf{C}}u}\omega_g
=
\frac{\sqrt{-1}}{2\pi}\bar{\partial}u. 
\]
\end{itemize}

\section{Proof of Theorem~\ref{main_thm}}
We first recall a standard identity for K\"ahler--Ricci solitons.
It is the K\"ahler counterpart of the standard conservation law for
gradient Ricci solitons; see \cite[Proposition 1.15]{CCGGIIKLLN07}; cf. also \cite[Section 20]{H95} for the steady case. 

\begin{lemma}\label{lem:conservation}
Let $g$ be a K\"ahler--Ricci soliton on a Fano manifold $X$. 
Then
\begin{equation}\label{eq:conservation}
|\bar\partial F_{g}|_{g}^{2} + F_{g} + s(\omega_{g}) = C
\end{equation}
for a constant $C \in \mathbf{R}$. 
\end{lemma}
\begin{proof}
For the reader's convenience, we include a short proof using Futaki's \emph{weighted Laplacian} (\cite{Fut87, Fut88}). 
Recall that the weighted Laplacian is defined by
\[
\Delta_{F_{g}}u = \Delta_{\omega_{g}}u - g(\bar\partial u, \overline{\bar\partial F_{g}}) = -\nabla^{i}\nabla_{i}u - \nabla^{i}u\nabla_{i}F_{g} 
\]
for any $u \in C^{\infty}(X, \mathbf{C})$. 

We use the standard fact that, if $\grad^{\mathbf{C}}_{\omega_{g}}u$ is holomorphic, then there exists a constant $c \in \mathbf{C}$ such that 
\[
\Delta_{F_{g}}(u + c) = u+c; 
\]
see, for example, \cite[Proposition 4.1]{Fut87}. 

Applying this to $u = F_{g}$, for which the constant $c$ is necessarily real, and using
\[
s(\omega_{g}) - n = -\Delta_{\omega_{g}}F_{g}, 
\]
we obtain
\[
 F_{g} + c = n - s(\omega_{g})-|\bar\partial F_{g}|_{g}^{2}, 
\]
which completes the proof. 
\end{proof}

We now prove Theorem~\ref{main_thm}. 
Let $g$ be an extremal K\"ahler--Ricci soliton on a Fano manifold $X$, and set 
\[
K_{u} \coloneqq J\grad_{g}u
\]
for any $u \in C^{\infty}(X, \mathbf{R})$. 
By assumption, both $K_{F_{g}}$ and $K_{s(\omega_{g})}$ are Killing vector fields. 
Let $h \coloneqq |\bar\partial F_{g}|_{g}^{2}$. 
Then, by Lemma~\ref{lem:conservation}, we have 
\[
h = C - F_{g} - s(\omega_{g}). 
\]
In particular, $K_{h}$ is also a Killing vector field. 

We next observe that $K_{F_{g}}$ and $K_{h}$ commute. 
Indeed, $K_{F_{g}}(F_{g})=0$, and since $K_{F_{g}}$ is Killing, it preserves the scalar curvature. 
Hence \eqref{eq:conservation} gives 
\[
K_{F_{g}}(h) = 0. 
\]
Moreover,
\[
i_{K_{h}}\omega_{g} = -\frac{1}{2\pi}dh\quad \text{and}\quad L_{K_{F_{g}}}\omega_{g} = 0. 
\]
Therefore
\[
i_{[K_{F_{g}}, K_{h}]}\omega_{g} 
= L_{K_{F_{g}}}i_{K_{h}}\omega_{g} - i_{K_{h}}L_{K_{F_{g}}}\omega_{g}
= -\frac{1}{2\pi}d(K_{F_{g}}(h))
= 0. 
\]
Since $\omega_g$ is nondegenerate, we obtain 
\[
[K_{F_{g}}, K_{h}] = 0. 
\]

Let $T$ be the closure in $\Isom(X,g)$ of the subgroup generated by the flows of $K_{F_{g}}$ and $K_{h}$. 
Since $K_{F_{g}}$ and $K_{h}$ commute, $T$ is a torus. 
Moreover, its action preserves $\omega_{g}$. 
Since $X$ is Fano, the Kodaira vanishing theorem and the Hodge decomposition imply that $H^{1}(X, \mathbf{R}) = \{0\}$. 
Thus the $T$-action is Hamiltonian. 
Let $\mu \colon X \to \Lie(T)^{\ast}$ be a moment map. 
Since $K_{h} \in \Lie(T)$ and $h$ is, up to a non-zero constant factor, a Hamiltonian potential of $K_{h}$, there exists an affine
function $\ell_{h} \colon \Lie(T)^{\ast} \to \mathbf{R}$ such that
\[
h = \ell_{h} \circ \mu. 
\]

Let $\operatorname{Fix}(T)$ be the fixed point set of the $T$-action.
If $x \in \Fix(T)$, then $K_{F_{g}}(x)=0$, and hence $dF_{g}(x) = 0$. 
Therefore, 
\[
h(x)=|\bar\partial F_{g}|_{g}^{2}(x) = 0. 
\]
It follows that $\ell_{h}$ vanishes on $\mu(\Fix(T))$. 
By the Atiyah--Guillemin--Sternberg convexity theorem (\cite{At82, GS82}), $\mu(X)$ is the convex hull of $\mu(\Fix(T))$. 
Since $\ell_{h}$ is affine, it follows that $\ell_{h}$ vanishes on $\mu(X)$. 
Hence
\[
h=|\bar\partial F_{g}|_{g}^{2} = 0 
\]
on $X$. 
In particular, $F_{g}$ is constant. 
By the normalization of the Ricci potential in \eqref{Ricci_potential}, we have $F_{g} = 0$, and consequently
\[
\Ric(\omega_{g})=\omega_{g}. 
\]
This completes the proof. 

\section{Concluding remarks}
We conclude this paper by pointing out that the preceding argument
has a Sasakian analogue. 
We refer to \cite{BG08} for basic material on Sasaki geometry, and to \cite{BGS08, FOW09} for the definitions of Sasaki--Ricci solitons and Sasaki-extremal metrics. 
Write $\eta$ for the contact form, $\xi$ for the Reeb vector field, and $g$ for the Sasaki metric on a compact Sasaki manifold $S$. 
The Killing vector fields appearing in the proof of Theorem~\ref{main_thm} have natural Sasakian counterparts, which are strict contact Killing vector fields. 
As in the proof above, they commute with each other, and their strict contactness implies that they commute with $\xi$. 
The closure in $\Isom(S, g)$ of the subgroup generated by their flows together with the Reeb flow is therefore a torus $T$. 

The Reeb foliation carries the natural transverse K\"ahler form $\frac{1}{2}d\eta$, and the functions $\frac{1}{2}\eta(Y)$, $Y \in \Lie(T)$, define the corresponding moment map.
By Ishida's transverse convexity theorem \cite[Theorem 1.2]{I17}, its image is the convex hull of the values attained on the common critical set of its components. 
At any point in this common critical set, every infinitesimal $T$-orbit is tangent to the Reeb direction. 
In particular, if $F_{g}$ denotes the transverse Ricci potential, the transverse part of the Sasakian counterpart of $K_{F_{g}}$ vanishes there, and hence so does $h=|\bar\partial_{B}F_{g}|^{2}$, where $\bar\partial_{B}$ denotes the basic $\bar\partial$-operator. 
Thus an argument parallel to the proof of Theorem~\ref{main_thm} yields the following result. 


\begin{theorem}
Every extremal Sasaki--Ricci soliton on a compact Sasaki manifold is Sasaki--Einstein. 
\end{theorem}

\begin{corollary}
Let $S$ be a compact Sasaki manifold. If $S$ does not admit a Sasaki--Einstein metric, then no Sasaki-extremal metric on $S$ can be transversely holomorphically isometric to a Sasaki--Ricci soliton via a Reeb-preserving diffeomorphism. 
\end{corollary}

\end{document}